\documentclass[english]{amsart}
\usepackage[T1]{fontenc}
\usepackage[latin9]{inputenc}
\usepackage{geometry}
\usepackage{amsmath}
\usepackage{amsthm}
\usepackage{amssymb}
\usepackage{graphicx}
\usepackage[numbers]{natbib}

\makeatletter
 \theoremstyle{definition}
 \newtheorem*{defn*}{\protect\definitionname}
\theoremstyle{plain}
\newtheorem{thm}{\protect\theoremname}
  \theoremstyle{plain}
  \newtheorem{conjecture}[thm]{\protect\conjecturename}
  \theoremstyle{plain}
  \newtheorem{prop}[thm]{\protect\propositionname}
  \theoremstyle{plain}
  \newtheorem{lem}[thm]{\protect\lemmaname}

\usepackage{amsmath}
\usepackage{amssymb}
\usepackage{wrapfig}

\makeatother

\usepackage{babel}
  \providecommand{\conjecturename}{Conjecture}
  \providecommand{\definitionname}{Definition}
  \providecommand{\lemmaname}{Lemma}
  \providecommand{\propositionname}{Proposition}
\providecommand{\theoremname}{Theorem}

\begin{document}

\calclayout
\title{Delay-Induced Switched States in a slow-fast system}


\author{Stefan Ruschel, Serhiy Yanchuk}


\maketitle
\begin{abstract}
We consider the two-component delay system $\varepsilon x^{\prime}(t)=-x(t)-y(t)+f(x(t-1)),$
$y^{\prime}(t)=\eta x(t)$ with small parameters $\varepsilon,\eta$,
and positive feedback function $f$. Previously, such systems have
been reported to model switching in optoelectronic experiments, where
each switching induces another one after approximately one delay time,
related to one round trip of the signal. In this paper, we study these
delay-induced switched states. We provide conditions for their existence
and show how the formal limits $\varepsilon\to0$ and/or $\eta\to0$
facilitate our understanding of this phenomenon. 
\end{abstract}

\section{Introduction}
Many nonlinear dynamical systems exhibit oscillations on different
time-scales characterized by interchanged phases of slow and fast
motion. Their analysis relies on the study of reduced systems, the
properties of which are then recombined using geometric singular perturbation
theory \cite{Kuehn2015}. Analogous tools are increasingly available
for functional differential equations \cite{Diekmann1995,Bates1998},
yet the effects of time-delay on multiple time-scale systems are largely
uncharted territory. It is very natural to observe relaxation oscillations
\cite{Fowler2002}, canard trajectories \cite{Campbell2009,Krupa2014,Souza2017}
and mixed mode oscillations \cite{Kouomou2005,TallaMbe2015}; phenomena
which are not delay-specific and extensively studied in the non-delayed
case \cite{Krupa2001,Desroches2012a}. Here, we study delay-induced
switched states as one further type of solution that naturally arises
in this set-up, yet cannot be found in the non-delayed case \cite{Weicker2012,Weicker2013,Larger2013,Larger2015,Erneux2016}. 

Let us consider the system 
\begin{align}
\varepsilon\dot{x}(t) & =-x(t)-y(t)+f(x(t-1)),\label{eq:def-x}\\
\dot{y}(t) & =\eta x(t),\label{eq:def-y}
\end{align}
where the dot denotes the derivative with respect to time. The parameters
$\varepsilon,\eta$ are assumed positive and sufficiently small so
that (\ref{eq:def-x})\textendash (\ref{eq:def-y}) can be regarded
as a singular perturbation problem acting on at least three different
timescales $\varepsilon,$ $1,$ and $1/\eta$. Figure~\ref{fig:A-solution}(a)
shows an example solution where the $x$-component switches back and forthbetween negative and positive values after approximately one delay time. The switches are sharp, on short intervals of order $\varepsilon$.
Between the switches, the solution changes slowly on time scale $1/\eta$.
Note that (\ref{eq:def-x})\textendash (\ref{eq:def-y}) can be transformed
to an equivalent system with short time delay $\eta,$ or long time
delay $1/\varepsilon$. 

This work is motivated by recent optoelectronic experiments \cite{Weicker2012,Weicker2013,Larger2013,Larger2015},
where delay-induced switched states were reported. Interestingly,
this set-up allows for switching between regular and irregular phases
reoccurring after approximately one delay time, see Fig.~\ref{fig:A-solution}(b)
for a numerically computed example trajectory showing this behavior,
together with the corresponding feedback function. In analogy to systems
of coupled oscillators \cite{Klinshov2017} and spatially extended
systems \cite{YanchukGiacomelli2017}, the authors refer to them
as \textit{chimera states }\cite{Abrams2004,Kuramoto2002} in Delay
Differential Equations. For numerical purposes, throughout the paper,
we choose Mackey-Glass like \cite{Mackey1977} feedback functions
$f_{1}(\xi)=1.2\xi/(1+\xi{}^{4})$ and $f_{2}(\xi)=2\xi/(1+(\xi-0.8)^{4})$,
which are capable to reproduce the qualitative dynamical behavior
of the physical set-up, yet are easier to manipulate. 

The paper is organized as follows: Sec.~\ref{sec:prelim} collects
preliminary results on scalar positive delayed feedback equations
and establishes the necessary vocabulary. After having these premises
set up, Sec.~\ref{sec:balance} lays out the conditions for the existence
of delay induced states. Sections.~\ref{sec:set-up}\textendash \ref{sec:coex}
contain more technical results. In particular, Sec.~\ref{sec:set-up}
introduces the mathematical framework for delay differential equations
and establish certain basic properties of (\ref{eq:def-x})\textendash (\ref{eq:def-y}).
In Sec.~\ref{sec:balance}, we establish the concept of a balance
point of delay-induced switched states. Section~\ref{sec:transition}
provides more insight into the mechanisms behind the emergence of
switched states close to the balance point. In Secs.~\ref{sec:profile}\textendash \ref{sec:transition},
we show that show that the formal limit system for $\varepsilon=0$
and $\eta=0$ explains the profile of delay-induced switched states
and another class of solutions that coexist. Our rigorous results
are given in the form of propositions and lemmas: the proofs are included in the supplementary material.
\begin{figure}[!]
\begin{centering}
\includegraphics[width=0.95\textwidth]{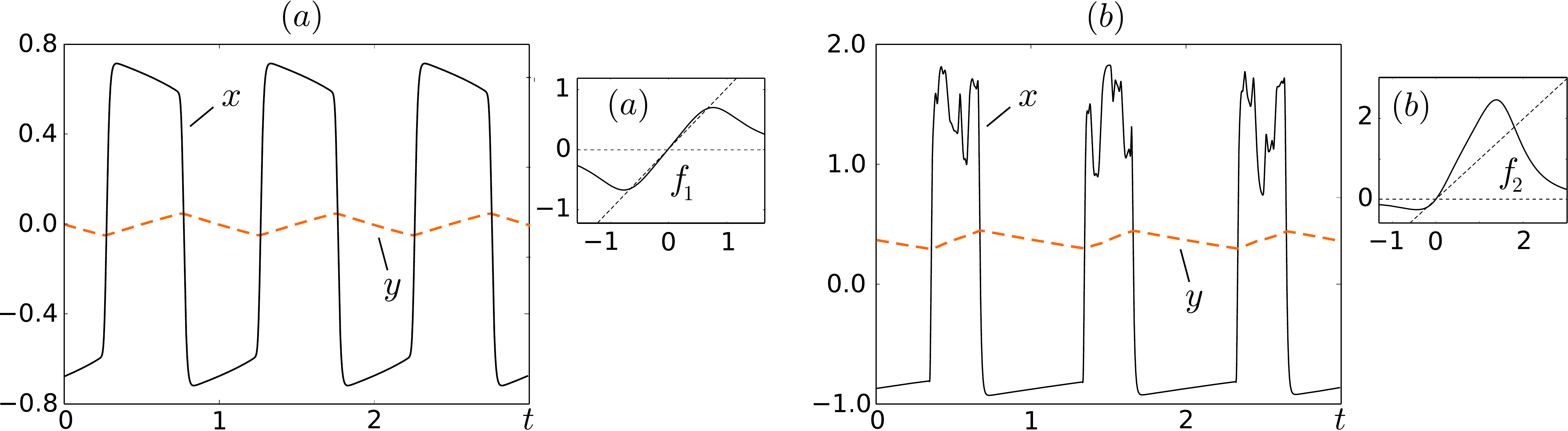}
\par\end{centering}
\caption{\label{fig:A-solution}(color online) Trajectories $(x,y)$ of (\ref{eq:def-x})\textendash (\ref{eq:def-y})
for parameter values $\varepsilon=0.01$, $\eta=0.3$, nonlinearities
(a)~$f_{1}(x)=1.2x/(1+x{}^{4})$, (b)~$f_{2}(x)=2x/(1+(x-0.8)^{4})$;
and suitable initial conditions. }
\end{figure}

\section{Positive Delayed Feedback \label{sec:prelim}}

To set the stage, allow us to back up a little and build some intuition.
Let us consider the simpler case of a scalar delayed feedback equation
\begin{align}
\varepsilon x^{\prime}(t) & =-x(t)+f(x(t-1)),\label{eq:def-x-0}
\end{align}
obtained from (\ref{eq:def-x})\textendash (\ref{eq:def-y}) by setting
$\eta$ and $y(0)$ to zero. Equations of type~(\ref{eq:def-x-0})
have been studied extensively in the literature. They are known to
cause a wide range of dynamical phenomena from oscillatory, to excitable
and chaotic behavior \cite{AnderHeiden1983}; in a variety of processes
such as physiological control systems \cite{Mackey1977}, nonlinear
optics \cite{Ikeda1979}, population dynamics \cite{Gurney1980}
and neuroscience \cite{Marcus1989}. 

For the purpose of this paper, we commit to a specific case, where
$f$ is continuously differentiable and 
\begin{itemize}
\item[\textbf{(H1)}]  $f(\xi)=\xi$ possesses exactly three solutions: $\xi^{-}<0$, $\xi=0$,
$\xi^{+}>0$,  and $f(\xi)\neq\xi$ otherwise.
\item[\textbf{(H2)}]  $f$ satisfies the positive feedback condition in some bounded but
sufficiently large interval: $\xi f(\xi)>0$ for all $\xi\in[\chi^{-},\chi^{+}]\backslash\{0\}$,
where $\chi^{+}=\max_{\xi\in I}f\left(\xi\right)$, $\chi^{-}=\min_{\xi\in I}f\left(\xi\right)$
and $I=[\xi^{-},\xi^{+}]$.  
\end{itemize}
Figure~\ref{fig:A-solution} shows examples of functions satisfying
properties (H1) \textendash{} (H2).

Let us consider an initial function $x_{0}(t)\in[\chi^{-},\chi^{+}]$,
for $t\in${[}-1,0{]}, taking positive as well as negative values.
For example, choose any slice of length one of the $x$-component
of the switching solution shown in Fig.~\ref{fig:A-solution}(b).
It can be shown that (H1) and (H2) imply $x(t)\in[\chi^{-},\chi^{+}]$
for all $t\geq0$ \cite{Rost2007}. Note that each switching between
two different phases contains a point, at which this function changes
sign. Under conditions (H1) and (H2), the solution to (\ref{eq:def-x-0})
has a remarkable property: The number of sign changes counted on intervals
of an appropriate length (that may be slightly larger than one) cannot
increase \cite{Cao1990,Mallet-Paret1996}. As a direct consequence,
the number of phases counted in this way cannot increase either. Their
width however, is subject to change; it might very well decrease until
a point at which its enclosing sign changes merge and the phase disappears.
We refer to this process as the \textit{coarsening} of a phase \cite{Giacomelli2012a}.
It is a spatiotemporal phenomenon best visualized in a spatiotemporal
plot \cite{YanchukGiacomelli2017}, see Fig.~\ref{fig:drift-charted-0}(a).
Here, the solution is sliced into subsequent segments of appropriate
length (about one), and the slice's values are then plotted with respect
to their counting number on the vertical axis, encoded in a color
map, see Ref. \cite{YanchukGiacomelli2017} for details. Fig.~\ref{fig:drift-charted-0}(a)
shows the temporal evolution of each phase in color/gray-scale, separated
by their sign changes in black. Interestingly, the amount with which
the position of a sign change alters from one segment to the next
seems independent of the phase width; as long as it is sufficiently
large \cite{Giacomelli2013}. We refer to this constant trend as
the \textit{drift} of this sign change. If this sign change has positive
slope, we refer to the corresponding drift as an $up$-drift $\delta^{\text{up}}$;
analogously we use the term \textit{down}-drift $\delta^{\text{down}}$.
Generally, up-drift and down-drift along a solution do not coincide,
leading to the coarsening of phases. 

This behavior is very well studied when $f$ is monotone \cite{Krisztin2008}.
In this case, the long term-behavior of solutions to (\ref{eq:def-x-0})
consists of only equilibrium solutions, unstable periodic orbits,
and homoclinic/heteroclinic connections between them \cite{Smith1987,Mallet-Paret1996a}.
Generically, solutions converge to equilibrium, yet with possibly
complicated transient behavior as the number of (unstable) periodic
solutions is large for small $\varepsilon$ \cite{Yanchuk2009}.
The dynamics is largely determined by the equilibria tugging at the
solution until the ``stronger'' wins. When $f$ is monotone, the
up- and down-drift along the solution can be computed explicitly,
see Ref. \cite{Grotta-Ragazzo1999,Grotta-Ragazzo2010} and references
therein. There, it is shown that two neighboring sign changes interact
weakly over exponentially small tails, and therefore, they may appear
independent numerically. The drift of the sign changes itself goes
to zero as $\varepsilon\to0$, whereas $d^{\text{up}}=\delta^{\text{up}}/\varepsilon$
and $d^{\text{down}}=\delta^{\text{up}}/\varepsilon$ convergence
to a bounded constant. If additionally $f$ is an odd function, the
situation is so well-adapted that up- and down-drift coincide to leading
order in $\varepsilon$, allowing for super-transient solutions with
transient times scaling as $e^{-c/\varepsilon}$, for some $c>0$.
We call any function $f$ with the property that up- and down-drift
coincide \textit{balanced }(not necessarily monotone or odd).

We want to remark that the case of negative delayed feedback has received
equal attention \cite{Fiedler1989}, and a lot of techniques for
positive delayed feed are motivated by the techniques developed for
the negative feedback case\cite{Chow1989,Lin1986,Nizette2003,Nizette2004, Wattis2017}. 

Note that our example $f=f_{1}$ in Fig.~\ref{fig:A-solution}(a)
is odd and restriction of $f_{1}$ to the values of the solution is
monotone. As a direct consequence, $f_{1}$ is balanced. It is easy
to show that continuous differentiability of $f$ implies that $f$,
when shifted along its graph by a small amount $f\mapsto f^{b}=f(\cdot+b)-f(b)$,
still satisfies (H1)-(H2). As a result $f_{1}$ has the following
property:
\begin{itemize}
\item[\textbf{(H3)}]  There exist $\beta\in\mathbb{R}$, and $r>0$ such that $f^{b}(\xi)=f(\xi+b)-f(b)$
satisfies (H1)\textendash (H2) for all $b\in[\beta-r,\beta+r]$ and
$f^{\beta}$ is balanced. We call $\beta$ the \textit{balance
point}.
\end{itemize}
In particular, $\beta=0$ for $f=f_1$.  Now obviously, our second example $f_{2}$ is neither odd
or monotone, nor is it balanced, as the corresponding solution displays
coarsening, see Fig.~\ref{fig:drift-charted-0}(a). Nevertheless,
it is straightforward to check that $f_{2}$ satisfies (H1)\textendash (H2).
In addition, we show numerically that $f_{2}$ satisfies (H3). 

Let us now vary $y(0)$ as a parameter and consider the equation
\begin{align}
\varepsilon x^{\prime}(t) & =-x(t)-y(0)+f(x(t-1)),\label{eq:def-x-0-1}
\end{align}
with $y(0)$ not necessarily zero. For sufficiently small $y(0)$,
the transformation $x^{b}(t)=x(t)-b,$ $f^{b}(\xi)=f(\xi+b)-f(b),$
where $b$ is the solution of $\xi=f(\xi)-y(0)$ with the smallest
modulus, takes Eq.~(\ref{eq:def-x-0-1}) into Eq.~(\ref{eq:def-x-0})
and our analysis above applies. 

We investigate numerically how the spatiotemporal behavior of (\ref{eq:def-x-0-1})
changes as we increase $y(0)$. In the example in Fig.~(\ref{fig:drift-charted-0})(a),
we observe that for $y(0)=0.2$ the difference of up- and down-drift
of the solution is smaller, and coarsening occurs at a later time.
This suggest that as we increase $y(0)$ the difference between up-
and down- drift further decreases. We show this numerically in Fig.~(\ref{fig:drift-charted-0})(c)
and (d). In fact, the values of the up- and down- drift change continuously
until they coincide at the balance point $y(0)=\beta\approx0.514$,
see Fig.~\ref{fig:drift-charted-0}(e). For larger values, the
order of the drifts is reversed and coarsening leads to another 'winning'
phase of the solution Fig.~\ref{fig:drift-charted-0}(d).

Note that throughout this process of the change of $y(0)$, the function
$f^{y(0)}$ satisfies (H1)-(H2), and as a result, $f_{2}$ satisfies
(H3). It is now intuitively clear, that close to the balance point
system (\ref{eq:def-x})\textendash (\ref{eq:def-y}), as a (small)
perturbation of (\ref{eq:def-x-0-1}) exhibits a long transient solution
with initial conditions close to the balance point. In fact, in the
next section we provide strong arguments in favor of these solution
not being mere transients anymore, but delay-induced switched states.
We conjecture that conditions (H1)-(H3) are necessary conditions on
$f$ to allow for such states.

\begin{figure}[t]
\centering{}\includegraphics[width=0.88\linewidth]{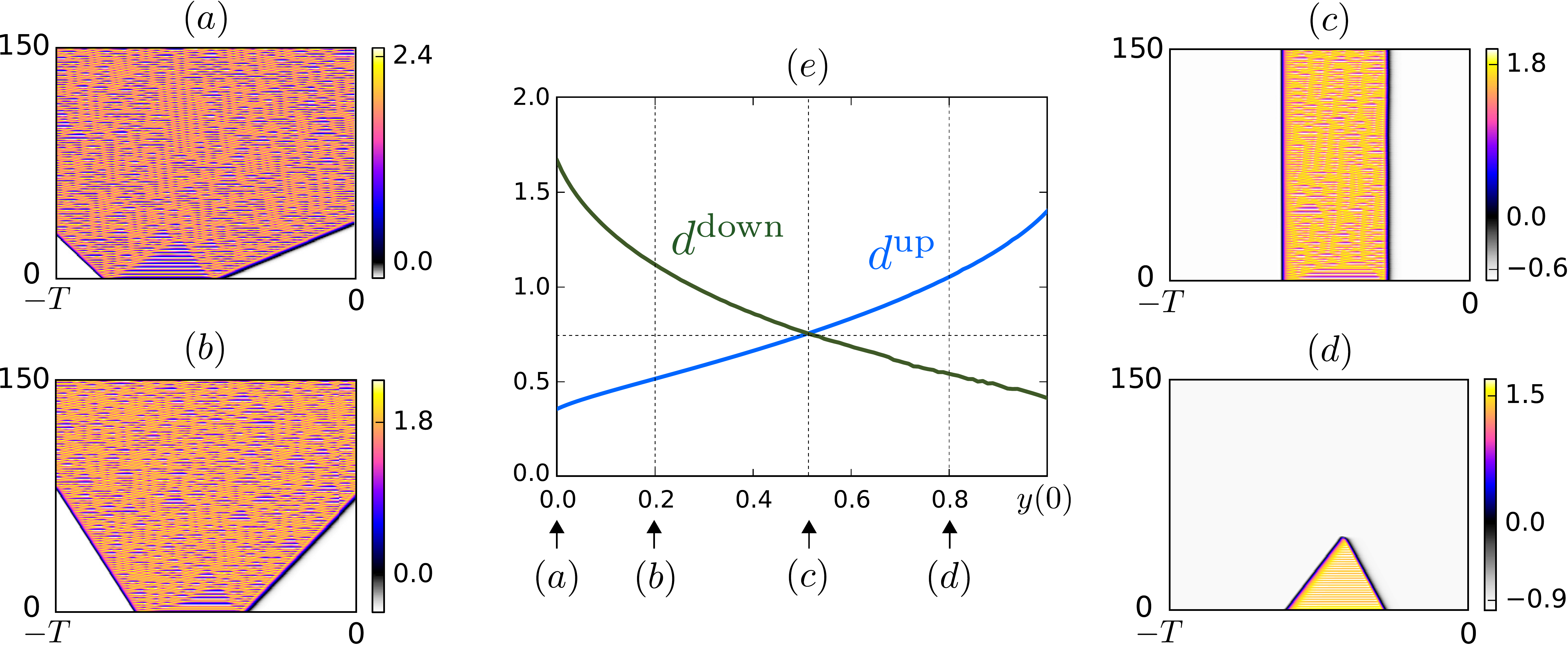}\caption{\label{fig:drift-charted-0} (color online) Numerical computation of the drift $d^{\text{up}}=\delta^{\text{up}}/\varepsilon$
of up (blue) and $d^{\text{down}}=\delta^{\text{down}}/\varepsilon$
of down (green) transitions layers for fixed $\varepsilon=0.01$,
$f_{2}(x)=2x/(1+(0.8-x)^{4})$, and initial function $x(t)=x_{0}(t)$,
$t\in[-1,0]$ with two sign changes (specified in the text). Panel
(a)-(d) show the spatiotemporal plot of the solution to (\ref{eq:def-x-0-1})
for for $T=1+0.755\varepsilon$ and different values of $y(0)$. Panel
(e) shows $d^{\text{up}}$ and $d^{\text{down}}$ with varying $y(0)\in[0,1]$. }
\end{figure}

\section{Main result: balance point and delay-induced
switched states \label{sec:balance}}

As we will show in Sec.~\ref{sec:set-up}, all solutions of (\ref{eq:def-x})\textendash (\ref{eq:def-y})
for $\eta>0$ that are not converging to or identically zero are oscillating,
i.e. its first component possesses infinitely many sign changes for
$t>0$ (Propositions~\ref{prop:zero-sol} and~\ref{prop:oscillatory}).
We are interested in oscillatory solutions that satisfy the following
definition motivated by the spatiotemporal plot in Fig.~\ref{fig:drift-charted-0}(c). 
\begin{defn*}[delay induced switched state]
\label{def:main-1} We refer to a solution $x\in C([-1,\infty),\mathbb{R})$
of (\ref{eq:def-x})\textendash (\ref{eq:def-y}) as a \emph{delay-induced
switched state}, if there exist $1\le T\le2$ and $0\le s\le1$, such
that $x(t)$ has a constant even number of sign changes on each interval
$[nT+s,\left(n+1\right)T+s]$ for all $n\geq0$. 
\end{defn*}
Additionally, we refer to $\delta=T-1$, as the \emph{drift} of a
delay induced switched state. It is easy to show that $\delta\to0$
as $\varepsilon\to0$. For this reason, in our numerical exploration,
we show the rescaled value $d=\delta/\varepsilon$ which converges
to a constant as $\varepsilon\to0$. See more details on the drift
in \cite{Giacomelli2012a,Giacomelli2013,YanchukGiacomelli2017}.

In a neighborhood of the balance point $\beta$ of system (\ref{eq:def-x-0-1}),
it is possible to obtain delay-induced switched states for (\ref{eq:def-x})-(\ref{eq:def-y}).
This can be explained as follows: If $y(t)<\beta$, the coarsening
of the solution leads to an increase of the average value $\frac{1}{T}\int_{-T}^{0}x(t+s)ds$
of $x(t)$ and therefore, to an increase of $y(t)=y(t-T)+\eta\int_{-T}^{0}x(t+s)ds$
towards the value $\beta$. Analogously, for $y(t)>\beta$, the average
$\frac{1}{T}\int_{-T}^{0}x(t+s)ds$ decreases as well as $y(t)$.
This allows for the ``dynamic stabilization'' of the delay-induced
switched states.

Figure~\ref{fig:drift-charted} shows numerical study of the up-
and down-drifts for the $x$-component of (\ref{eq:def-x})-(\ref{eq:def-y})
for different values of $\eta$ and $y(0)$. As the spatiotemporal
plots in Figs.~\ref{fig:drift-charted}(a,b,c1,c2,c3,d) show, the
coarsening may still occur for certain parameters as in Figs.~\ref{fig:drift-charted}(a,b,d),
while switched states are observed for the other parameters as in
Figs.~\ref{fig:drift-charted}(c1,c2,c3). The conditions, where switched
states exist, i.e. where the up- and down-drifts coincide $\delta^{+}=\text{\ensuremath{\delta}}^{-}=\delta$,
are shown as the gray region in Figs.~\ref{fig:drift-charted}(e,f).
The existence region has the form of a cone attached to a balance
point $y(0)=\beta$ at $\eta=0$. That is, the transient switched state for
$\eta=0$ exists only in the balance point, and any deviation of the
initial condition $y(0)$ from $\beta$ would lead to the coarsening.
With increasing $\eta$, the allowed range of $y(0)$, where delay-induced switched
states exist, increases. This indicates that such states become stable
for $\eta>0$, and in order to reach them, one should initiate $y(0)$
in the neighborhood of the balance point $\beta$, while the size
of this neighborhood grows with $\eta$. We want to remark that this mechanism resembles a locking cone with respect to the delayed self-feedback, and indeed, similar solutions have been observed in forced systems \cite{Lefebvre2012}.

\begin{figure}[!]
	\centering{}\includegraphics[width=0.9\linewidth]{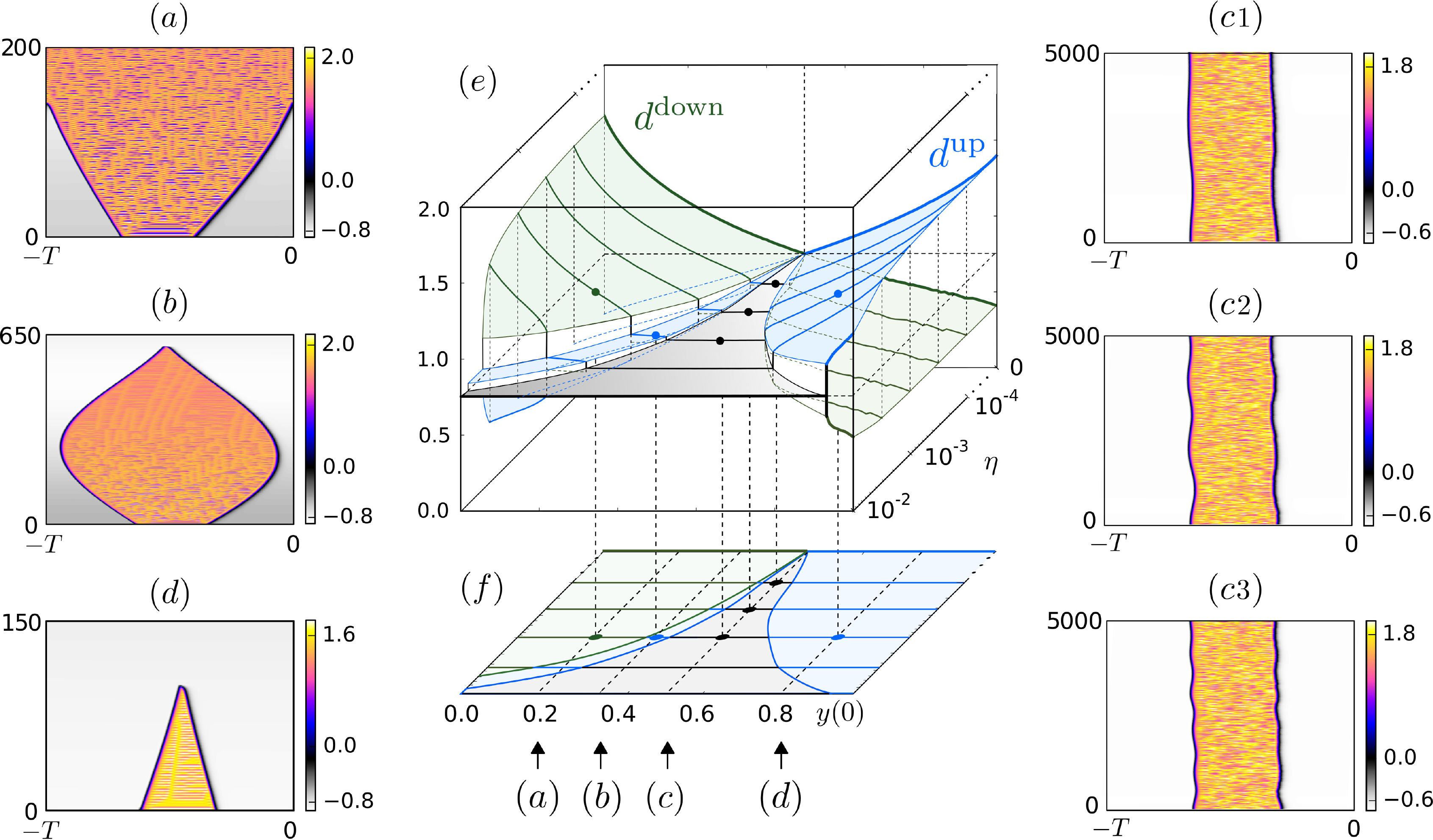}\caption{\label{fig:drift-charted}(color online) Up-drift $d^{\text{up}}=\delta^{\text{up}}/\varepsilon$
		(blue) and down-drift $d^{\text{down}}=\delta^{\text{down}}/\varepsilon$
		(green) of solutions for fixed $\varepsilon=0.01$, $f_{2}(x)=2x/(1+(0.8-x)^{4})$,
		$T=1+0.514\varepsilon$ and initial function $x(t)=x_{0}(t)$, $t\in[-1,0]$
		with two sign changes (specified in the text). Panels (a),(b),(c1),(d)
		show spatiotemporal plots of the solution for $\eta=0.001$ and (a)
		$y(0)=0.2$, (b) $y(0)=0.33$, (c1) $y(0)=0.514$ and (d) $y(0)=0.8$.
		Panels (c2),(c3) show spatiotemporal plots of the solution for $y(0)=0.514$
		and (c2) $\eta=0.0003$ and (c3) $y(0)=0.0001$, (c1) $y(0)=0.514$
		and (d) $y(0)=0.8$. Values of $d^{\text{up}}$ and $d^{\text{down}}$
		varying the feedback parameter $\eta\in[0,0.01]$ and initial condition
		$y(0)\in[0,1]$. Panel (e): numerically computed parameter regions
		(and individual points), where $d^{\text{up}}$ and $d^{\text{down}}$
		coincide are shown in gray. Panel (f): projection (from above) of
		Panel (e) onto the $(y(0),\eta)$-plane. Colors encode: $d^{\text{up}}<d^{\text{down}}$
		(blue), $d^{\text{up}}=d^{\text{down}}$ (gray), $d^{\text{up}}>d^{\text{down}}$
		(green).}
\end{figure}

Our rigorous analysis in the following sections shows that the solution
of (\ref{eq:def-x})-(\ref{eq:def-y}) can be obtained as a perturbation
of (\ref{eq:def-x-0-1}) (Lemma~\ref{lem:lemma-L-L-1}) and that
this perturbation contracts the distance of subsequent sign changes
to a critical distance (implied by Lemma~\ref{lem:L-zeros}). The
``largeness'' of this correcting effect strongly depends on $\eta$.
Too far from the balance point, it is too weak to prevent the coarsening
of a phase, see Fig.~\ref{fig:drift-charted}(a). Similarly, the
time scale of this process is of importance. Closer to the balance
point the influence of $\eta$ is strong enough to prevent the coarsening
of one phase, but then it can lead to the coarsening towards the other,
see Fig.~\ref{fig:drift-charted}(b). We also show that the profile
of the switched state can be obtained from a reduced system (Proposition
\ref{prop:per-1}).

Supported by our rigorous analysis of the ``stabilization'' mechanism
described above as well as numerical evidences in Fig.~\ref{fig:drift-charted}
we make the following conjecture. 
\begin{conjecture}
\label{thm:main-1-1} Let $f\in C^{1}(\mathbb{R})$ satisfy $(H1)$\textendash $(H3)$
with balance point $\beta$. Then, for all $\varepsilon>0$ and $\eta>0$
sufficiently small, all initial conditions with $y(0)$ sufficiently
close to $\beta$, and $x(\theta)\in[\chi^{-},\chi^{+}]$, for $-1\le\theta\le0$,
with at least one sign changes lead to a delay-induced switched state. 
\end{conjecture}

An initial condition without sign changes leads to oscillatory solutions
of relaxation type, similar to that described in \cite{Kouomou2005,TallaMbe2015}.
We have included a numerical exploration and discussion of this type
of oscillatory behavior in Sec.~\ref{sec:coex}.

\section{Basic properties of system (\ref{eq:def-x})\textendash (\ref{eq:def-y})
\label{sec:set-up}}

Equations (\ref{eq:def-x})\textendash (\ref{eq:def-y}) define a
(semi-)dynamical system with phase space $C:=C\left(\left[-1,0\right],\mathbb{R}^{2}\right)$,
that is the Banach space of continuous functions taking values in
$\mathbb{R}^{2}$, with norm $\left\Vert \phi\right\Vert =\sup_{\theta\in\left[-1,0\right]}\left|\phi\left(\theta\right)\right|$,
$|\cdot|$ being the Euclidean norm in $\mathbb{R}^{2}$. Given an
initial function $\phi\in C$, the solution $z(t,\phi)=(x(t,\phi),y(t,\phi)),$
$t\geq0,$ of the initial value problem
\begin{align}
\frac{dz}{dt}(t) & =A^{\varepsilon,\eta}z(t)+F\left(z(t-1)\right),\,z_{0}=\phi,\label{eq:def-z}
\end{align}
where $A^{\varepsilon,\eta}=\left(\begin{array}{cc}
-\frac{1}{\varepsilon} & -\frac{1}{\varepsilon}\\
\eta & 0
\end{array}\right),\,F(\zeta)=\left(\begin{array}{c}
f(\zeta)\\
0
\end{array}\right),$ exists and is unique \cite{Hale1993}. Here, we use the standard
notation $z_{t}(\phi)=z(t+\theta;\phi),$ for $\theta\in[-1,0],$
respectively for $x$ and $y$, to refer to the state of the system
$z_{t}\in C$ at time $t\in[0,\infty)$, whenever it is convenient.
Note that in order to have full invariant manifolds theory available,
we would need to extend this concept of solution to the state space
$\mathbb{R}^{2}\times L^{\infty}(\left[-1,0\right],\mathbb{R}^{2})$
\cite{Diekmann1995}, which is not strictly necessary in our case.
The solution defines a $C^{1}$ semi-flow $\Phi_{t}:\phi\mapsto x_{t}(\phi)$
on $C$ \cite{Hale1993}. Recall that the \textit{$\omega$-limit
set} $\omega(\phi)$ of $\phi\in C$ under the semi-flow $\Phi_{t}$
is defined to be 
\[
\omega(\phi)=\{\psi\in C\ |\ \Phi_{t_{n}}\phi\to\psi\mbox{ for some sequence }t_{n}\to\infty\}.
\]
If $z_{t}(\phi)$ exists for all $t$, $z_{t}(\phi)$ is $C^{1}$
with a uniform bound on its derivatives for all $t\ge t_{0}+1$. By
the Arzela-Ascoli Theorem then, $\omega(\phi)$ is nonempty and compact
in $C$ (with its $C^{0}$ norm). We discuss some of the objects in
$\omega(\phi)$ of (\ref{eq:def-x})\textendash (\ref{eq:def-y}).
As $\Phi_{t}$ is $C^{1}$ we rely on the principle of linearized
stability to determine the asymptotic behavior of initial conditions
close to $\omega(\phi)$. 

For our further analysis, it is convenient to write the solution in
terms of the solution operator that takes solution segments that have
length $T=1+\delta$ slightly larger than one from one to the next.
Consider the solution segment $\phi_{n}(\theta):=z(nT+\theta)\in C_{T}^{1}:=C^{1}([-T,0],\mathbb{R}^{2}),$
$n\in\mathbb{N}$ and $\theta\in[-T,0],$ $1\leq T<2$. Then, (\ref{eq:def-z})
is equivalent to
\begin{align}
\frac{d}{d\theta}\phi_{n+1}(\theta) & =A^{\varepsilon,\eta}\phi_{n+1}(\theta)+\begin{cases}
F\left(\phi_{n}(\theta+\delta)\right), & -T<\theta<-\delta,\\
F\left(\phi_{n+1}(\theta-1)\right), & -\delta<\theta<0,
\end{cases}\label{eq:def-phi}\\ 
\phi_{n+1}(-T) & =\phi_{n}(0)\qquad\text{ for all }n\in\mathbb{N},\,\left.\phi_{0}\right|_{[-1,0]}=z_{0}\in C([-1,0],\mathbb{R}^{n}).\nonumber 
\end{align}
Eq.~(\ref{eq:def-phi}) defines a map given by $M_{T}^{\varepsilon,\eta}:C_{T}\to C_{T},$
$\psi\mapsto M_{T}^{\eta}\psi$ with
\begin{align}
\left[M_{T}^{\varepsilon,\eta}\psi\right](\theta) & = e^{A^{\varepsilon,\eta}(\theta+T)}\psi(-T)+\int_{-T}^{\min\left\{ \theta,-\delta\right\} }e^{A^{\varepsilon,\eta}(\theta-s)}F\left(\psi(s+\delta)\right)\text{d}s\label{eq:def-M}\\ 
 & +\int_{\min\left\{ \theta-1,-T\right\} }^{\theta-1}e^{A^{\varepsilon,\eta}(\theta-1-s)}F\left(e^{A^{\varepsilon,\eta}(s+T)}\psi(-T)+\int_{-T}^{s}e^{A^{\varepsilon,\eta}(\theta-\tilde{s})}\psi(\tilde{s}+\delta)\text{d}\tilde{s}\right)\text{d}s,\nonumber
\end{align}
which can be easily seen using the variations of constants formula.
The following two propositions collect basic properties of (\ref{eq:def-x})\textendash (\ref{eq:def-y}).
\begin{prop}[zero solution]
\label{prop:zero-sol} Let $\varepsilon,\eta>0$ and $f$ satisfy
(H1)-(H2). Then, 
\begin{enumerate}
\item[(i)] (\ref{eq:def-x})\textendash (\ref{eq:def-y}) has the unique equilibrium
solution $z\equiv0$. 
\item[(ii)] there exist \textup{$\bar{\varepsilon}>0$,} such that $z\equiv0$
is unstable for all \textup{$\varepsilon<\bar{\varepsilon}$}. At
\textup{$\varepsilon=\bar{\varepsilon}$}, the equilibrium undergoes
a Hopf bifurcation. More specifically, there exists $\varepsilon_{k}\to0$
as $k\to\infty$, such that for every $k$ there is a complex conjugated
pair of eigenvalues crossing the imaginary axis with nonzero speed.
\end{enumerate}
\end{prop}

\begin{proof}
	This is a direct result of the asymptotic spectral properties of steady
	states of delay equations with one single discrete delay \cite{Lichtner2011}.
	We will not repeat the full analysis here. It is easy to see that
	system~(\ref{eq:def-x})\textendash (\ref{eq:def-y}) has the unique
	equilibrium solution $z\equiv0$. The behavior of solutions in its
	neighborhood is determined by the characteristic growth rates, which
	are solutions to the corresponding characteristic equation $0=\lambda\left(\varepsilon\lambda+1-f^{\prime}(0)e^{-\lambda}\right)+\eta$.
	Using the change $\mu=\varepsilon\lambda$, this equation reads as
	$0=\mu\left(\mu+1-f^{\prime}(0)Y(\mu/\varepsilon)\right)+\varepsilon\eta$,
	where $Y=e^{-\mu/\varepsilon}$, and the analysis of \cite{Lichtner2011}
	directly applies. In particular,the spectrum for small $\varepsilon$
	can be characterized by its asymptotic properties. The strong asymptotic
	spectrum is absent. The weak asymptotic spectrum is given by $\mu_{\text{weak}}(\varphi)=\varepsilon\gamma(\varphi)+i\varphi,$
	where $\gamma(\varphi)=-\frac{1}{2}\log|Y(i\varphi)|$ and 
	\[
	|Y(i\varphi)|=\frac{\left(\varepsilon\eta-\varphi^{2}\right)^{2}+\varphi^{2}}{f^{\prime}(0)^{2}\varphi^{2}}.
	\]
	(H1)-(H2) imply $|f^{\prime}(0)|>1$, such that $\max_{\varphi}\gamma(\varphi)>0$,
	i.e. the weak spectrum is unstable. Our assertion follows directly
	from \cite{Lichtner2011}.
\end{proof}

We have established that the zero equilibrium is unstable, with an
arbitrary large dimension of the unstable manifold. As a result, we
expect the deviation from it to grow exponentially for a generic small
perturbation. This destabilization is due to a cascade of Hopf bifurcations
as $\varepsilon\to0$, so one might expect periodic solutions to be
the simplest objects that occur. It is intuitively clear that for
$\eta>0$, the solution is forced to oscillate. A precise statement
is given in the following Proposition.
\begin{prop}[oscillatory solutions]
\label{prop:oscillatory}Let $\varepsilon,\eta>0$ and $f$ satisfy
(H1)-(H2). Then any solution $z(t,\phi)$, which does not converge
to zero, possesses the following properties: 
\begin{enumerate}
\item[(i)]  There exists $0<t<\infty$, such that $z^{1}(t,\phi)=0$. 
\item[(ii)]  The first component $z^{1}$of $z$ has infinitely many sign changes. 
\item[(ii)]  For all $t_{1}<t_{2}$, such that \textup{$z^{2}(t_{1},\phi)=z^{2}(t_{2},\phi)$},
it holds that$\int_{t_{1}}^{t_{2}}z^{1}(s,\phi)ds=0$.
\end{enumerate}
\end{prop}

\begin{proof}
	(i) Suppose that the assertion is false, i.e. $z^{1}(t,\phi)>0$ for
	all $t\in[t_{0},\infty)$; the other case of different sign is completely
	analogous. Since $z^{1}(t,\phi)$ is strictly bounded away from zero,
	$z^{2}(t,\phi)$ is strictly monotonously increasing with $z^{2}(t,\phi)=\eta\int_{0}^{t}z^{1}(s,\phi)ds$.
	Thus, there exists $t_{1}$ such that for all $t>t_{1}$, $(z^{1})^{\prime}(t,\phi)<-z^{1}(t,\phi)-z^{2}(t,\phi)+M<-c$
	for some $c>0$. This implies that $z^{1}(t,\phi)$ changes sign leading
	to a contradiction. (ii) is a direct consequence of (i). (iii) Direct
	integration of \eqref{eq:def-y} gives the result, $0=z^{2}(t_{2},\phi)-z^{2}(t_{1},\phi)=\eta\int_{t_{1}}^{t_{2}}z^{1}(s,\phi)ds$.
\end{proof}

As the zero solution is unstable, we have thus established that every solution, except the zero solution and the solutions that lie on the stable manifold of the zero solution are oscillatory for all $t\geq0$.


\section{Stabilization mechanism of delay-induced switched states \label{sec:transition}}

As we have indicated in Sec.~\ref{sec:balance}, initial conditions
'close' the balance point lead to delay-induced switched states, whereas
those too far from the balance point, may appear as delay-induced
switched states at first, yet coarsen after some possibly long transient
time, and as a result no delay-induced switching can be observed anymore.
This can be explained by a competition of the fast $\varepsilon-$dynamics
and slow $\eta$-dynamics. In order to investigate these processes
independently, we 'split off' the $\eta$-part of the dynamics. In
doing so, we provide the conceptional mechanism for the stabilization
of delay-induced switched states; the more involved arguments being
given as self-contained lemmas. Recall from Sec.~\ref{sec:set-up}
that the solution to (\ref{eq:def-x})\textendash (\ref{eq:def-y})
can be written in integral form
\begin{align}
x(t) =~ & x(t_{0})e^{-(t-t_{0})/\varepsilon} \nonumber \\
&+\int_{t_{0}}^{t}\frac{e^{-(t-s)/\varepsilon}}{\varepsilon}\left(f(x(s-1))-y(t_{0})\right)ds-\eta\int_{t_{0}}^{t}\frac{e^{-(t-s)/\varepsilon}}{\varepsilon}\left(\int_{t_{0}}^{s}x(\tilde{s})d\tilde{s}\right)ds,\label{eq:x-slow-bound-1}\\
y^{\prime}(t) = ~ & y(t_{0})+\eta\int_{t_{0}}^{t}x(s)ds,\label{eq:y-slow-bound-1}
\end{align}
which can be readily checked using the variation of constants formula.
Our motivation is simple, find a transformation such that the resulting
solution is independent of $\eta$. A closer look at system (\ref{eq:x-slow-bound-1})-(\ref{eq:y-slow-bound-1})
suggests to subtract the respective last terms $-\eta\int_{t_{0}}^{t}\frac{e^{-(t-s)/\varepsilon}}{\varepsilon}\left(\int_{t_{0}}^{s}x(\tilde{s})d\tilde{s}\right)ds$
and $\eta\int_{t_{0}}^{t}x(s)ds$ in (\ref{eq:x-slow-bound-1})-(\ref{eq:y-slow-bound-1}).
If we assume a delay-induced switched state $\psi$ with length T,
then this transformation takes the form of a linear operator $L_{T}^{\varepsilon,\eta}:C_{T}^{1}\to C_{T}^{1},\psi\mapsto L_{T}^{\varepsilon,\eta}\psi,$
with 
\begin{equation}
[L_{T}^{\varepsilon,\eta}\phi](\theta)=\phi(\theta)+\eta\begin{pmatrix}\frac{1}{\varepsilon}\int_{-T}^{\theta}e^{-\frac{(\theta-s)}{\varepsilon}}\left(\int_{-T}^{s}\phi^{1}(\tilde{s})d\tilde{s}\right)ds\\
-\int_{-T}^{\theta}\phi^{1}(s)ds
\end{pmatrix}.\label{eq:def-L}
\end{equation}
The next Lemma collects several properties of $L_{T}^{\varepsilon,\eta}$;
the proof of which is contained in the supplementary material. Recall
that $M_{T}^{\varepsilon,\eta}$ is the solution operator that maps
segments of delay-induced switched states of length $T$ to the next. 
\begin{lem}
\label{lem:lemma-L-L-1}Let $\varepsilon,\eta,T>0,\,\text{\ensuremath{\phi}}\in C_{T}^{1}$
be fixed, and $L_{T}^{\varepsilon,\eta}$ be defined as in (\ref{eq:def-L}).
Then,
\begin{enumerate}
\item[(i)] $L_{T}^{\varepsilon,\eta}\phi$ is a $\eta$-small perturbation of
\textup{$\phi$}. 
\item[(ii)] $L_{T}^{\varepsilon,\eta}$ is one-to-one. $K_{T}^{\varepsilon,\eta}=(L_{T}^{\varepsilon,\eta})^{-1}$
has the form
\begin{equation}
(K_{T}^{\varepsilon,\eta}\phi)(\theta)=\phi(\theta)+\int_{-T}^{\theta}\left(A^{\varepsilon,\eta}e^{A^{\varepsilon,\eta}(\theta-s)}-e^{A^{\varepsilon,\eta}(\theta-s)}A^{\varepsilon,0}\right)\phi(s)ds.\label{eq:def-L-1}
\end{equation}
\item[(iii)] \textup{$L_{T}^{\varepsilon,\eta}M_{T}^{\varepsilon,\eta}=M_{T}^{\varepsilon,0}$,}
and \textup{$M_{T}^{\varepsilon,\eta}=K_{T}^{\varepsilon,\eta}M_{T}^{\varepsilon,0}$}.
\end{enumerate}
\end{lem}

\begin{proof}
Let $\phi\in C_{T}^{1}$ and denote $\bar{\phi}=L_{T}^{\varepsilon,\eta}\phi$.
(i) One readily checks the estimates $|\bar{\phi}^{2}(\theta)-\phi^{2}(\theta)|=\eta|\int_{-T}^{\theta}\phi^{1}(s)ds|\leq\eta T|\phi^{1}|\leq2\eta\left\Vert \phi\right\Vert $
and 
\[
|\bar{\phi}^{1}(\theta)-\phi^{1}(\theta)|=\frac{\eta}{\varepsilon}|\int_{-T}^{\theta}e^{-\frac{(\theta-s)}{\varepsilon}}\left(\int_{-T}^{s}\phi^{1}(\tilde{s})d\tilde{s}\right)ds|\leq\frac{2\eta}{\varepsilon}\left\Vert \phi\right\Vert \int_{-T}^{\theta}e^{-\frac{(\theta-s)}{\varepsilon}}ds\leq2\eta\left\Vert \phi\right\Vert ,
\]
such that $\left\Vert L_{T}^{\varepsilon,\eta}\phi-\phi\right\Vert \leq3\eta\left\Vert \phi\right\Vert $
where we have used $1\leq T<2$. Note that $(\bar{\phi}^{1})^{\prime}(\theta)=(\phi^{1})^{\prime}(\theta)-(\bar{\phi}^{1}(\theta)-\phi^{1}(\theta)+\bar{\phi}^{2}(\theta)-\phi^{2}(\theta)),$
$(\bar{\phi}^{2})^{\prime}(\theta)=(\phi^{2})^{\prime}(\theta)+\eta\phi^{1}(\theta)$
such that $\left\Vert \bar{\phi}{}^{\prime}-\phi{}^{\prime}\right\Vert <2\left\Vert \bar{\phi}-\phi\right\Vert +\eta T\left\Vert \phi\right\Vert <8\eta\left\Vert \phi\right\Vert $.
Thus, $L_{T}^{\varepsilon,\eta}-id$ is uniformly bounded by $8\eta$.
(ii) Taking the derivative of $L_{T}^{\varepsilon,\eta}\phi$ one
observes that $\bar{\phi}\in C_{T}^{1},$ with
\[
\frac{d}{d\theta}\bar{\phi}(\theta)=A^{\varepsilon,0}\bar{\phi}(\theta)+\frac{d}{d\theta}\phi(\theta)-A^{\varepsilon,\eta}\phi(\theta),
\]
and $A^{\varepsilon,\eta}$ is defined as in Eq.~(\ref{eq:def-z}).
This is ODE that can be solved explicitly to recover Eq.~(\ref{eq:def-L}).
On the other hand, it can be rearranged as an IVP problem for $\phi$,
i.e.
\begin{eqnarray}
\frac{d}{d\theta}\phi(\theta) & = & A^{\varepsilon,\eta}\phi(\theta)+\frac{d}{d\theta}\bar{\phi}(\theta)-A^{\varepsilon,0}\bar{\phi}(\theta),\label{eq:def-L-1-deriv}\\
\phi(-T) & = & \bar{\phi}(-T).\nonumber 
\end{eqnarray}
In this way $\phi$ is uniquely determined given $\bar{\phi}\in C_{T}^{1}$
and the explicit form of $K_{T}^{\varepsilon,\eta}$ is readily obtained
by the variation of constants formula. (iii) Let $\phi_{0}\in C_{T}^{1},$
and consider the sequence $(\bar{\phi}_{n})_{n}$ where $\bar{\phi}_{n}=L_{T}^{\varepsilon,\eta}\phi_{n}$
and $\phi_{n+1}=M_{T}^{\varepsilon,\eta}\phi_{n}$ for all $n\geq0$.
Then, it is easy to check that $\bar{\phi}_{n}$ satisfies 
\begin{eqnarray*}
\frac{d}{d\theta}\bar{\phi}_{n}(\theta) & = & A^{\varepsilon,0}\bar{\phi}_{n}(\theta)+\begin{cases}
F\left(\phi_{n-1}(\theta+\delta)\right), & -T<\theta<-\delta,\\
F\left(\bar{\phi}_{n}(\theta-1)\right), & -\delta<\theta<0,
\end{cases},\\
\bar{\phi}_{n}(-T) & = & \phi_{n-1}(0),
\end{eqnarray*}
such that $L_{T}^{\varepsilon,\eta}M_{T}^{\varepsilon,\eta}\phi_{n-1}=L_{T}^{\varepsilon,\eta}\phi_{n}=M_{T}^{\varepsilon,0}\phi_{n-1}$.
This is independent of $\phi_{n-1}$ and therefore $L_{T}^{\varepsilon,\eta}M_{T}^{\varepsilon,\eta}=M_{T}^{\varepsilon,0}$.
The second assertions follows immediately using (ii).
\end{proof}

As a result, we obtain a decomposition of the solution operator $M_{T}^{\varepsilon,\eta}=K_{T}^{\varepsilon,\eta}M_{T}^{\varepsilon,0}$
into two parts. We use now that we have qualitative information about
$M_{T}^{\varepsilon,0}$, see Sec.~\ref{sec:balance}. In particular,
the number of sign changes of a given function $\text{\ensuremath{\phi}}\in C_{T}^{1}$
does not increase, and the amount of which sign changes drift can
be read of by our numerical studies for $\eta=0.$ In addition, we
can show how $K_{T}^{\varepsilon,\eta}$ effects the position of sign
changes of a function $\phi\in C_{T}$ along the interval $[-T,0]$.
We do so indirectly, via studying its inverse $L_{T}^{\varepsilon,\eta}$
. The results are collected in the following Lemma. For brevity we say that a function $\phi$ has the property
\begin{itemize}
\item[\textbf{$\boldsymbol{(SC_{N})}$}]  The first component $\phi^{1}$ of $\phi$ has an even number of
sign changes and no other zeros. More specifically, there exist $N\in\mathbb{N}$,
and a series of sign changes $(\theta^{i})$, $0\leq i\leq2N$, of
$\phi^{1}$ such that $\phi^{1}(\theta^{i})=0$, $(\phi^{1})^{\prime}(\theta^{i})\neq0$
and $\phi^{1}$ has no other zeros.
\end{itemize}
\begin{lem}[dynamics of sign changes under $L_{T}^{\varepsilon,\eta}$]
\label{lem:L-zeros}Let $\varepsilon,\eta,T>0$ be fixed. Let $\phi\in C_{T}^{1}$
satisfy $(SC_{N})$ and consider $\bar{\phi}=L_{T}^{\varepsilon,\eta}\phi$.
Then the following hold true.
\begin{enumerate}
\item[(i)] There exists $\bar{\eta}>0,$ depending on $\phi$, such that for
all $\eta<\bar{\eta}$, $\bar{\phi}$ also satisfies $(SC_{N})$.
In particular, for each sign change $\theta$ of $\phi^{1}$, there
exists exactly one sign change $\bar{\theta}$ of $\bar{\phi}^{1}$
with $|\theta-\bar{\theta}|=\mathcal{O}(\eta)$ and $\bar{\phi}^{1}$
has no other zeros. 
\item[(ii)] If $\phi\in C_{T}^{2}$, there exists $\bar{\eta}^{\prime}>0,$ depending
on $\phi$, such that for all $\eta<\bar{\eta}^{\prime}$, the sign
changes $\bar{\theta}$ of $\bar{\phi}^{1}$ are given by
\begin{equation}
\bar{\theta}=\theta-\frac{\eta}{\varepsilon(\phi^{1})^{\prime}(\theta)}\int_{-T}^{\theta}e^{-\frac{(\theta-s)}{\varepsilon}}\left[\int_{-T}^{s}\phi^{1}(\tilde{s})d\tilde{s}\right]ds+\mathcal{O}(\eta^{2}),\label{eq:L-zeros}
\end{equation}
 where $\theta$ is a sign change of $\phi$. 
\item[(iii)]  Assume as in (ii). There exist $\bar{\varepsilon}>0$ such that
if $\varepsilon<\bar{\varepsilon}$, then for each subsequent pair
of sign changes \textup{$\theta^{i},\theta^{i+1},\,i\leq2N$} there
exist unique (up to order $\eta^{2}$) $\vartheta^{i}\geq0$, depending
on $\phi$, such that if $|\theta^{i}-\theta^{i+1}|<\vartheta^{i}+\mathcal{O}(\eta^{2})$
$($respectively $|\theta^{i}-\theta^{i+1}|>\vartheta^{i}+\mathcal{O}(\eta^{2}))$,
then $|\bar{\theta}^{i}-\bar{\theta}^{i+1}|<|\theta^{i}-\theta^{i+1}|$
$($respectively $|\bar{\theta}^{i}-\bar{\theta}^{i+1}|>|\theta^{i}-\theta^{i+1}|)$.
In particular, $\vartheta^{i}>0$ (up to order $\eta^{2}$), if and
only if
\[
\frac{1}{(\phi^{1})^{\prime}(\theta^{i})}\left(\int_{-T}^{\theta^{i}}e^{-\frac{(\theta^{i}-s)}{\varepsilon}}\left[\int_{-T}^{s}\phi^{1}(\tilde{s})d\tilde{s}\right]ds\right)<0.
\]
\item[(iv)]  assume as in (iii). If $\phi$ is $(SC_{2})$ with $-T<\theta^{1}<\theta^{2}<0$,
$\vartheta^{1}>0$ (up to an error of order $\eta^{2}$), and $\vartheta^{1}$
is implicitly given (up to an error of order $\eta^{2}$) by 
\[
\int_{\theta^{1}}^{\theta^{1}+\vartheta^{1}}\phi^{1}(s)ds=\left(\frac{(\phi^{1})^{\prime}(\theta^{2})}{(\phi^{1})^{\prime}(\theta^{1})}-1\right)\int_{-T}^{\theta^{1}}\phi^{1}(s)ds.
\]
\end{enumerate}
\end{lem}

\begin{proof}
The proof relies on Newton's method to approximate the sign changes
of $\bar{\phi}$ for $\eta\ll1$. Note that $\phi(-T)=\bar{\phi}(-T)$
such that without loss of generality, we may assume that $\phi^{1}(-T)\neq0$,
i.e. $-T$ is not a sign change of $\phi^{1}$ or $\bar{\phi}^{1}$.
Throughout the proof, we restrict ourselves to $\phi^{1}(-T)<0$;
the proof of the other case is completely analogous.

(i) We show that for sufficiently small $\eta$, $\bar{\phi}$ also
has property $(SC_{N})$. From $(SC_{N})$ of $\phi$, there exist
$\theta^{i}$, $0\leq i\leq2N$, $\phi^{1}(\theta^{i})=0$ and intervals
$(\alpha_{i},\beta_{i})\ni\theta^{i},$ such that $(-1)^{i-1}\phi^{1}(\alpha_{i})<0<(-1)^{i-1}\phi^{1}(\beta_{i})$
and $\left.\phi^{1}\right|_{[\alpha_{i},\beta_{i}]}$ is strictly
monotone. Let $I=\bigcup_{i}[\alpha_{i},\beta_{i}]$ and denote $c=\inf_{\theta\in[-T,0]\backslash I}|\phi^{1}(\theta)|$,
$c^{\prime}=\min_{\theta\in I}|(\bar{\phi}^{1})^{\prime}(\theta)|$.
Using Lemma~\ref{lem:lemma-L-L-1}, $\bar{\phi}$ is a continuously
differentiable, $\eta$-small perturbation of $\phi$ with $C^{1}-$uniform
bound $8\eta\left\Vert \phi\right\Vert $. Thus, we may choose $\bar{\eta}=\min\left\{ c,c^{\prime},\min{}_{i}\{|\phi^{1}(\alpha_{i})|\},\min_{i}\{|\phi^{1}(\beta_{i})|\}\right\} /(8\eta\left\Vert \phi\right\Vert )>0,$
such that $(-1)^{i-1}\bar{\phi}^{1}(\alpha_{i})<0<(-1)^{i-1}\bar{\phi}^{1}(\beta_{i})$,
$\left.\bar{\phi}^{1}\right|_{[\alpha_{i},\beta_{i}]}$ is strictly
monotone with $\bar{c}^{\prime}=\min_{\theta\in I}|(\bar{\phi}^{1})^{\prime}(\theta)|>0$,
and $\bar{\phi}^{1}$ is bounded away from zero outside of $I$ for
all $\eta<\bar{\eta}$. Then, by the Intermediate Value Theorem, there
exist unique (since $\bar{\phi}^{1}$ is strictly monotone there)
$\bar{\theta}^{i}\in[\alpha_{i},\beta_{i}]$ such that $\bar{\phi}^{1}(\bar{\theta}^{i})=0$,
$(\bar{\phi}^{1})^{\prime}(\bar{\theta}^{i})\neq0$, and there are
no other zeros of $\bar{\phi}^{1}$. 

Using the Mean Value Theorem for each $i$, there exist $\xi^{i}\in[\min\{\theta^{i},\bar{\theta}^{i}\},\max\{\theta^{i},\bar{\theta}^{i}\}]\subset[\alpha_{i},\beta_{i}]$
such that $(\bar{\phi}^{1})^{\prime}(\xi^{i})=\left(\bar{\phi}^{1}(\bar{\theta}^{i})-\bar{\phi}^{1}(\theta^{i})\right)/\left(\bar{\theta}^{i}-\theta^{i}\right)$.
Consequently, 
\[
|\bar{\theta}^{i}-\theta^{i}|\leq\frac{|\bar{\phi}^{1}(\bar{\theta}^{i})-\bar{\phi}^{1}(\theta^{i})|}{\min_{\theta\in I}|(\bar{\phi}^{1})^{\prime}(\theta)|}=\frac{1}{\bar{c}^{\prime}}|\phi^{1}(\theta^{i})-\bar{\phi}^{1}(\theta^{i})|\leq b\eta
\]
for all $i$, where $b=3T\left\Vert \phi\right\Vert /\bar{c}^{\prime}$.
Here, we have used $\bar{\phi}^{1}(\bar{\theta}^{i})=0=\phi^{1}(\theta^{i})$
for all $i$. 

(ii) Let $\eta<\bar{\eta}$. Fix a sign change $\theta$ of $\phi^{1}$
and the corresponding sign change $\bar{\theta}$ of $\bar{\phi}^{1}$
with $|\bar{\theta}-\theta|=\mathcal{O}(\eta)$. We show that under
certain conditions to be specified $\bar{\theta}$ is determined up
to a bounded error of order $\eta^{2}$ by only one step of Newton's
method with the initial guess $\theta$. Recall from (i) that there
exists an interval $[\alpha,\beta]\in\{\theta,\bar{\theta}\}$ such
that the derivatives of $\phi^{1}$ and $\bar{\phi}^{1}$ are bounded
away from $0$ by a constant $b^{\prime}$. For example, $b^{\prime}=\min\left\{ c^{\prime},\bar{c}^{\prime}\right\} $,
where $c^{\prime},\bar{c}^{\prime}$ defined as in (i). Let us denote
$\theta^{\prime}$ the first iterate of the Newton approximation step
$\theta\mapsto\theta^{\prime}$, then
\begin{eqnarray*}
\theta^{\prime} & = & \theta-\frac{\bar{\phi}^{1}(\theta)}{(\bar{\phi}^{1})^{\prime}(\theta)},\\
 & = & \theta-\frac{\frac{\eta}{(\phi^{1})^{\prime}(\theta)}\frac{1}{\varepsilon}\int_{-T}^{\theta_{i}}e^{-\frac{(\theta_{i}-s)}{\varepsilon}}\left[\int_{-T}^{s}\phi^{1}(\tilde{s})d\tilde{s}\right]ds}{1-\frac{\eta}{(\phi^{1})^{\prime}(\theta)}\left(\frac{1}{\varepsilon}\int_{-T}^{\theta_{i}}e^{-\frac{(\theta_{i}-s)}{\varepsilon}}\left[\int_{-T}^{s}\phi^{1}(\tilde{s})d\tilde{s}\right]ds+\int_{-T}^{\theta}\phi^{1}(s)ds\right)}.
\end{eqnarray*}
Here, we used the definition of $\bar{\phi}^{1}(\theta)$, the fact
$\phi^{1}(\theta)=0$, and multiplied the numerator and the denominator
by $1/(\phi^{1})^{\prime}(\theta)$. For sufficiently small $\eta<b^{\prime}/(2T\left\Vert \phi\right\Vert )$,
we can expand the denominator into a geometric series such that $\theta^{\prime}$
satisfies Eq.~(\ref{eq:L-zeros}) (with $\bar{\theta}$ replaced
by $\theta^{\prime}$). $\bar{\phi}\in C_{T}^{2}$ satisfies the estimate
$\left\Vert \bar{\phi}^{\prime\prime}-\phi^{\prime\prime}\right\Vert <10\eta\left\Vert \phi\right\Vert $
such that $\left\Vert \bar{\phi}^{\prime\prime}\right\Vert <\left\Vert \phi^{\prime\prime}\right\Vert +10\eta\left\Vert \phi\right\Vert =:b^{\prime\prime}$.
Then, the error $|\bar{\theta}-\theta^{\prime}|$ of the Newton approximation
step $\theta\mapsto\theta^{\prime}$ is bounded as
\[
|\bar{\theta}-\theta^{\prime}|<\frac{b^{\prime\prime}}{2b^{\prime}}|\bar{\theta}-\theta|^{2}\left(<\eta\frac{b^{\prime\prime}b}{2b^{\prime}}|\bar{\theta}-\theta|\right),
\]
where for the last inequality, we have used $|\bar{\theta}-\theta|<\eta b$
from (i). It is a well known fact that if $\eta\frac{b^{\prime\prime}b}{2b^{\prime}}<1$
here, Newton's method converges, and quadratically so, such that $|\bar{\theta}-\theta^{\prime}|=\mathcal{O}(|\bar{\theta}-\theta|^{2})=\mathcal{O}(\eta^{2})$.
As a result, Eq.~(\ref{eq:L-zeros}) holds for $\eta<\tilde{\eta}=\min\left\{ \bar{\eta},b^{\prime}/(2T\left\Vert \phi\right\Vert ),2b^{\prime}/b^{\prime\prime}b\right\} $. 

(iii) Consider two subsequent sign changes $\theta^{i},\theta^{i+1}$
as discussed in (i) and assume as in (ii). We show that as we apply
$L_{T}^{\varepsilon,\eta}$, for sufficiently small $\varepsilon$,
the distance between the sign changes $\theta^{i+1}-\theta^{i}$ decreases
(increases), if it is smaller (larger) than a given ``critical''
distance $\vartheta^{i}$, depending on $\phi$ and $\theta^{i}$,
such that $|\bar{\theta}^{i+1}-\bar{\theta}^{i}|<|\theta^{i+1}-\theta^{i}|$
(respectively $|\bar{\theta}^{i+1}-\bar{\theta}^{i}|>|\theta^{i+1}-\theta^{i}|$).
Let $(\phi^{1})^{\prime}(\theta^{i})>0$, $(\phi^{1})^{\prime}(\theta^{i+1})<0$
be fixed such that $\phi^{1}(\theta)>0$ for all $\theta\in(\theta^{i},\theta^{i+1})$
and vary $\vartheta=\theta^{i+1}-\theta^{i}>0$ as an independant
variable; the proof of the case with signes exchanged is completely
analogous. To increase readability, we introduce new variables $Y(-T,\theta)=\int_{-T}^{\theta}\phi^{1}(s)ds$
and $X(-T,\theta)=\frac{1}{\varepsilon}\int_{-T}^{\theta}e^{-\frac{(\theta-s)}{\varepsilon}}Y(-T,s)ds$
such that at $\theta^{i+1},$ we can express the value of $X(-T,\theta^{i+1})$
as
\[
X(-T,\theta^{i}+\vartheta)=X(\theta^{i},\theta^{i}+\vartheta)+Y(-T,\theta^{i})+e^{-\frac{\vartheta}{\varepsilon}}(X(-T,\theta^{i})-Y(-T,\theta^{i})).
\]
We note that $|X(-T,\theta)-Y(-T,\theta)|\leq\varepsilon T\left\Vert \phi\right\Vert $
for all $\theta\in[-T,0]$. Using the above expression and Eq.~\ref{eq:L-zeros},
the difference $\bar{\theta}^{i+1}-\bar{\theta}^{i}$ can be written
as
\[
\bar{\theta}^{i+1}-\bar{\theta}^{i}=\vartheta+\frac{a_{i}\eta}{(\phi^{1})^{\prime}(\theta^{i})}Z^{i}(\vartheta)+\mathcal{O}(\eta^{2}),
\]
 where $a_{i}=-(\phi^{1})^{\prime}(\theta^{i})/(\phi^{1})^{\prime}(\theta^{i+1})>0$
and
\begin{eqnarray*}
Z^{i}(\vartheta) & = & X(-T,\theta^{i}+\vartheta)+\frac{1}{a_{i}}X(-T,\theta^{i}),\\
 & = & X(\theta^{i},\theta^{i}+\vartheta)+Y(-T,\theta^{i})+\frac{1}{a_{i}}X(-T,\theta^{i})+e^{-\frac{\vartheta}{\varepsilon}}(X(-T,\theta^{i})-Y(-T,\theta^{i})).
\end{eqnarray*}
Thus, up to an error of order $\eta^{2}$, $\bar{\theta}^{i+1}-\bar{\theta}^{i}$
is smaller (larger) than $\vartheta$, if $Z^{i}(\vartheta)<0$ (respectively
$Z^{i}(\vartheta)>0$). Clearly, $Y(\theta^{i},\theta^{i}+\vartheta)=\int_{\theta^{i}}^{\theta^{i}+\vartheta}\phi^{1}(s)ds$
and therefore $X(\theta^{i},\theta^{i}+\vartheta)=\frac{1}{\varepsilon}\int_{\theta^{i}}^{\theta^{i}+\vartheta}e^{-\frac{(\theta^{i}+\vartheta-s)}{\varepsilon}}Y(\theta^{i},\theta^{i}+s)ds$
are monotonously increasing with $Y(\theta^{i},\theta^{i})=0$ and
$X(\theta^{i},\theta^{i})=0$. We remark that in the limit $\varepsilon\to0$,
statement (iii) is trivial now: Consider
\[
\lim_{\varepsilon\to0}Z^{i}(\vartheta)=Y(\theta^{i},\theta^{i}+\vartheta)+(1+\frac{1}{a_{i}})Y(-T,\theta^{i}),
\]
and observe that the second term is constant with $1+1/a_{i}>0$.
Thus, if $Y(-T,\theta^{i})<0$, there exist a unique $\vartheta^{i}>0$
such that $Z^{i}(\vartheta^{i})=0$ and $\bar{\theta}^{i+1}-\bar{\theta}^{i}=\vartheta^{i}+\mathcal{O}(\eta^{2})$.
If $Y(\theta^{i})\geq0$, then $\lim_{\varepsilon\to0}Z^{i}(\vartheta)>0$
for all $\vartheta>0$ and we may choose $\vartheta^{i}=0$.

For $0<\varepsilon<\bar{\varepsilon}=\eta/(T\left\Vert \phi\right\Vert )$,
we have that $|X(-T,\theta)-Y(-T,\theta)|<\eta$ and thus,
\begin{eqnarray*}
Z^{i}(\vartheta) & = & X(\theta^{i},\theta^{i}+\vartheta)+(1+\frac{1}{a_{i}})X(-T,\theta^{i})+\mathcal{O}(\eta).
\end{eqnarray*}
We can apply the same reasoning as before. The remaining terms of
order $\eta$ affect the difference $\bar{\theta}^{i+1}-\bar{\theta}^{i}$
only of order $\eta^{2}$ and can be neglected. In particular, $\vartheta^{i}>0$
(up to an error of order $\eta^{2}$), if and only $X(-T,\theta^{i})<0$.

(iv) Let $X(-T,\theta)$ and $Y(-T,\theta)$ be defined as in (iii).
By assumption $(\phi^{1})^{\prime}(\theta_{1})>0$ and $\phi^{1}(\theta)<0$
for all $\theta\in[-T,\theta_{1})$ such that $X(-T,\theta_{1})<0$.
Then, $Z^{1}(\vartheta^{1})=0$ implies $\vartheta^{1}>0$. The implicit
expression for $\vartheta^{1}$ follows directly from the condition
$Z^{1}(\vartheta^{1})=0$ using that $||X-Y||<\eta$. 
\end{proof}

As an immediate consequence, for sufficiently small $\eta$, $K_{T}^{\varepsilon,\eta}$
preserves number of zeros and contracts the distance between two zeros
to some positive distance. Thus, if for a given $\varepsilon$, we
are sufficiently close to the balance point, the individual drift
of the sign changes can be mitigated by the action of $K_{T}^{\varepsilon,\eta}$.
Naturally, as $K_{T}^{\varepsilon,\eta}$ influences the profile of
delay induced states, we expect small variations in the 'shape' of
delay-induced switched states as opposed to the solution of (\ref{eq:def-x})\textendash (\ref{eq:def-y})
for $\eta=0.$ The following section discusses this effect in more
detail.


\section{Profiles of delay-induced switched states\label{sec:profile}}

We reason that for a given function $f$ the ``shape'' of a delay-induced
switched state of (\ref{eq:def-x})\textendash (\ref{eq:def-y}) can
be qualitatively determined from the formal limit $\varepsilon=0$.
The motivation is not unlike the scalar case of delayed feedback:
If $\varepsilon x'(t)$ acts as a small perturbation, then the solution
to (\ref{eq:def-x})\textendash (\ref{eq:def-y}) can be sufficiently
well approximated by the reduced system
\begin{align}
x(t) & =-y(t)+f(x(t-1)),\label{eq:def-x-sing}\\
y^{\prime}(t) & =\eta x(t).\label{eq:def-y-sing}
\end{align}
Only when the derivative is large with respect to $1/\varepsilon$
this viewpoint becomes inadequate. Using the variations of constants
formula, (\ref{eq:def-x})\textendash (\ref{eq:def-y}) satisfies
\begin{eqnarray}
x(t) & = & -y(t)+f(x(t-1))+\left(x(t_0)+y(t_0)-f(x(t_0)-1)\right)e^{-(t-t_{0})/\varepsilon} \nonumber \\
&&+\int_{t_{0}}^{t}e^{-(t-s)/\varepsilon}\left[f^{\prime}(x(s-1))x'(s-1)-\eta x(s)\right]ds\label{eq:x-slow-bound}\\
y^{\prime}(t) & = & \eta x(t),\label{eq:y-slow-bound}
\end{eqnarray}
The integral representation (\ref{eq:x-slow-bound})\textendash (\ref{eq:y-slow-bound})
of (\ref{eq:def-x})\textendash (\ref{eq:def-y}) allows for the following
insight: As long as $|x^{\prime}(t)|$ is bounded and independent
of $\varepsilon$, i.e. $|x^{\prime}(t)|\ll C/\varepsilon$ for some
$C<\infty$, it immediately follows that $x(t)=-y(t)+f(x(t-1))+\mathcal{O}(\varepsilon)$
for all $t$. On the other hand, if the derivative is large, yet only
on a short interval $(a,b)\subset[-1,0],$ with $|b-a|\to0$ as $\varepsilon\to0$,
and $c/\varepsilon<|x^{\prime}(\theta)|<C/\varepsilon$ for some $c<C;$
then $x(t)=-y(t)+f(x(t-1))+\mathcal{O}\left(1\right)$ for $\theta\in(a,b)$.
So the derivative may act as a large perturbation within $(a,b)$.
However then, for $\theta\in(b,1]$ one can easily show that this
perturbation decays fast again as $x(t)=-y(t)+f(x(t-1))+\mathcal{O}\left(e^{-(t-b)/\varepsilon}\right)+\mathcal{O}\left(\varepsilon\right)$.
This results in the solution segment on $[t,t+1]$ to appear as closely
related to the solution segment on $[t-1,t]$, but possibly shifted
with respect to that segment.

Note that although $\varepsilon$ is supposed to be small, (\ref{eq:def-x})\textendash (\ref{eq:def-y})
cannot (or rather should not) be considered a 'small' perturbation
of (\ref{eq:x-slow-bound})\textendash (\ref{eq:y-slow-bound}) as
(i) $\varepsilon x^{\prime}(t)$ ought not be small, and (ii) the
nature of our problem changes from a Delay Differential Equation to
a Delay Difference Equation. This singular perturbation point of view
comes with several such technical difficulties.

Here, we focus on the analysis of the reduced system (\ref{eq:def-x-sing})\textendash (\ref{eq:def-y-sing}).
The following proposition states that (\ref{eq:def-x-sing})\textendash (\ref{eq:def-y-sing})
exhibits a period-1 solution, and sufficient conditions for it to
be unstable.
\begin{prop}[Period-$1$ solutions]
\label{prop:per-1} Let $\eta>0$ be sufficiently small. Then, 
\begin{itemize}
\item[(i)] There exists a period-$1$ solution $(x,y)$ of (\ref{eq:def-x-sing})\textendash (\ref{eq:def-y-sing}).
\item[(ii)]  $x$ satisfies $\int_{-1}^{0}x(t+s)ds=0$ for all $t.$
\item[(iii)] If the derivative $x^{\prime}(t)$ is defined, then
it satisfies
\begin{equation}
x^{\prime}(t)=\eta x(t)/(f^{\text{\ensuremath{\prime}}}(x(t))-1).\label{eq:x-manifold}
\end{equation}
\item[(iv)] If \textup{$|f^{\text{\ensuremath{\prime}}}(x(t))|>1$ }for some
$t\in[-1,0],$ then $x$ is unstable.
\end{itemize}
\end{prop}

\begin{proof}
We refer to functions $x$ (of bounded variation) and $y$ (differentiable)
that satisfy (\ref{eq:def-x-sing})\textendash (\ref{eq:def-y-sing}),
and $x(t)=x(t-1)$, $y(t)=y(t-1)$ for all $t$, as period$-1$ solutions
of (\ref{eq:def-x-sing})\textendash (\ref{eq:def-y-sing}). A solution
to (\ref{eq:def-x-sing})\textendash (\ref{eq:def-y-sing}) satisfies
the functional equation 
\begin{align}
x(t) & =-y(t-1)-\eta\int_{t-1}^{t}x(s)ds+f(x(t-1)),\label{eq:def-x-sing-2}\\
y(t) & =y(t-1)+\eta\int_{t-1}^{t}x(s)ds.\label{eq:def-y-sing-2}
\end{align}
(\ref{eq:def-x-sing-2})\textendash (\ref{eq:def-y-sing-2}) follow
from straightforward integration of (\ref{eq:def-y-sing}). We construct
period$-1$ solutions of (\ref{eq:def-x-sing})\textendash (\ref{eq:def-y-sing})
using (\ref{eq:def-x-sing-2})\textendash (\ref{eq:def-y-sing-2}).
If such a solution exists, it immediately satisfies 
\begin{equation}
0=y(t)-y(t-1)=\eta\int_{-1}^{0}x_{t}(s)ds\label{eq:red-x-per-1}
\end{equation}
 for all $t$ and the values of $x$ and $y$ are fixed as $x_{k}(\theta)=x_{k-1}(\theta)=x(\theta)$
and $y_{k}(\theta)=y_{k-1}(\theta)=y(\theta)$ for all $k\in\mathbb{Z}$
and $\theta\in[-1,0]$. The resulting system
\begin{align}
x_{k}(\theta) & =-y(\theta)+f(x_{k}(\theta)),\label{eq:def-x-sing-1}\\
y_{k}(\theta) & =y(\theta),\label{eq:def-y-sing-1}
\end{align}
is decoupled in each ``point in space'' $\theta$ and $y(\theta)=y(0)+\eta\int_{-1}^{\theta}x_{0}(s)ds$
is to be considered as a parameter. We prescribe a fixed number of
sign changes $-1<\theta^{1}<\theta^{2}<\dots<\theta^{2N}<0,$ $2N\in\mathbb{N}$,
and in addition, choose initial values $\chi^{i},\,1\leq i\leq2N,$
$\chi^{i}\in\mathbb{R}$, with $\text{sgn}(\chi^{i})=-\text{sgn}(\chi^{i+1})$,
$\chi^{2N+1}=\chi^{1}$, and $f(\chi^{i})\neq1$. For $\theta\in[\theta^{i},\theta^{i+1})$
(with $\theta^{2N+1}=\theta^{1}$), $x_{0}(\theta)$ can be defined
as the solution of the Ordinary Differential Equation
\begin{equation}
x_{0}^{\prime}(\theta)=\eta x_{0}(\theta)/(f^{\text{\ensuremath{\prime}}}(x_{0}(\theta))-1),\label{eq:per-1-constr}
\end{equation}
with $x_{0}(\theta^{i})=\chi^{i}$. Here $\eta$ has to be small enough
such that $f^{\text{\ensuremath{\prime}}}(x_{0}(\theta))\neq1$ for
all $\theta\in[-1,0]$. As a result, $x_{0}$ is piecewise differentiable
(continuously differentiable from the right) with discontinuities
at $\theta^{i}$, and $x_{0}$ satisfies $x_{0}(0)=x_{0}(-1)$. Eq.~(\ref{eq:per-1-constr})
can be integrated such that $x_{0}(\theta)=y_{0}(\theta)+f(x_{0}(\theta))$,
where $y_{0}(\theta):=y_{0}(-1)+\eta\int_{-1}^{\theta}x_{0}(s)ds$
and $y_{0}(-1):=-x_{0}(-1)+f(x_{0}(-1)$. As $(\theta^{i}),$$(\chi^{i})$
were arbitrary, we may choose them such that $\int_{-1}^{0}x_{0}(s)ds=0$
and $y_{0}(0)=y_{0}(-1)$. As a result, the function $(x,y)$ defined
as $x_{k}(\theta)=x_{0}(\theta),\,y_{k}(\theta)=y_{0}(\theta),$ for
all $k\in\mathbb{Z}$ and $\theta\in[-1,0]$ is a 1-periodic solution
of (\ref{eq:def-x-sing-1})\textendash (\ref{eq:def-y-sing-1}). By
construction, $x(t)=x(t-1)$ and $y^{\prime}(t)=\eta x(t)$ for all
$t$, such that $(x,y)$ satisfies (\ref{eq:def-x-sing})\textendash (\ref{eq:def-y-sing}).
Thus follows (i) and (ii) is obvious from Eq.~(\ref{eq:red-x-per-1});
(iii) from Eq.~(\ref{eq:per-1-constr}).

For (iiv), we consider the time evolution $(\xi^{s},v^{s})$ of a
point perturbation $(\xi_{0}^{s},0)$ along the the period$-1$ solution
$(x_{0},y_{0})$, where $\xi_{0}^{s}(s)=1$ and $\xi_{0}^{s}(\theta)=0$
for all $\theta\in[-1,0]/\{s\}.$ One can easily show that $\xi^{s}$
satisfies 
\begin{align}
\xi_{n}^{s}(s) & =f^{\prime}(x_{0}(s))\xi_{n-1}^{s}(s),\label{eq:def-x-sing-2-1}
\end{align}
and $\xi_{n}^{s}(\theta)=0$ for all $n\in\mathbb{N}$ and $\theta\in[-1,0]/\{s\}.$
Here, we have used that $\eta\int_{-1}^{\theta}\xi_{n}^{s}(s)ds=0$
and $v_{n}^{s}(\theta)=v_{n-1}^{s}(\theta)=0$ for all $n\in\mathbb{N}$
and $\theta\in[-1,0].$ As a result, the perturbation $\xi_{0}^{s}$
grows if and only if $|f^{\prime}(x_{0}(s))|>1.$ This condition is
sufficient for $x_{0}$ to be unstable. 
\end{proof}
Our numerics shows that the solutions from the Proposition~\ref{prop:per-1}
are related to the delay-induced switched states in (\ref{eq:def-x})\textendash (\ref{eq:def-y}). Figure~\ref{fig:per-1}
displays the solution to (\ref{eq:def-x-sing})\textendash (\ref{eq:def-y-sing})
for two example functions $f$ and suitable initial conditions. Fig.~\ref{fig:per-1}(a)
shows a period-$1$ solution that switches between two distinct sets
of values with opposite sign, along which it changes regularly. In
fact, it changes exponentially with a small rate, and Prop.~\ref{prop:per-1}(iii)
determines this rate to be $\eta/(f^{\text{\ensuremath{\prime}}}(x(t))-1)$.
We remark that for small values of $\eta$, this trend can be approximated
as 
\[
x(t+\Delta t)=x(t)e^{\int_{t}^{t+\Delta t}\eta/(f^{\prime}(x(s))-1)ds}\approx x(t)+\mathcal{O}(\eta\Delta t),
\]
so that along such regular ``plateaus'' it appears to be linear
in $\eta$. Here, we assumed that $f^{\prime}(x(t))-1$ is uniformly
bounded away from zero. In Fig.~\ref{fig:per-1}(a), the parameters
are chosen such that $f^{\text{\ensuremath{\prime}}}(x(t))<1$ along
the solution. If this condition is not met for some values of $x$,
the corresponding period-$1$ solution is unstable, and the numerically
computed solution appears irregular for those values, see Fig.~\ref{fig:per-1}(b)
and compare Prop.~\ref{prop:per-1}(iii). This qualitatively resembles
delay-induced switched states with incoherent part in (\ref{eq:def-x})\textendash (\ref{eq:def-y}),
although the solution is more erratic than for positive $\varepsilon$,
compare Fig.~\ref{fig:A-solution}. We attribute this to the ``smoothing''
effect of the term $\varepsilon x^{\prime}$ in (\ref{eq:def-x})\textendash (\ref{eq:def-y}),
which is not present in (\ref{eq:def-x-sing})\textendash (\ref{eq:def-y-sing}).

The solutions shown in Fig.~\ref{fig:per-1} correspond to well chosen
initial conditions satisfying the integral property in Prop.~\ref{prop:per-1}(ii).
For strictly positive (or analogously strictly negative) initial conditions
the solution to (\ref{eq:def-x-sing})\textendash (\ref{eq:def-y-sing})
does not resemble a delay-induced switched state. As a result, for
such initial conditions, we do not expect delay-induced states in
(\ref{eq:def-x})\textendash (\ref{eq:def-y}). The next section discusses
this fact in more detail.

\begin{figure}[!]
	\centering{}\includegraphics[width=0.9\linewidth]{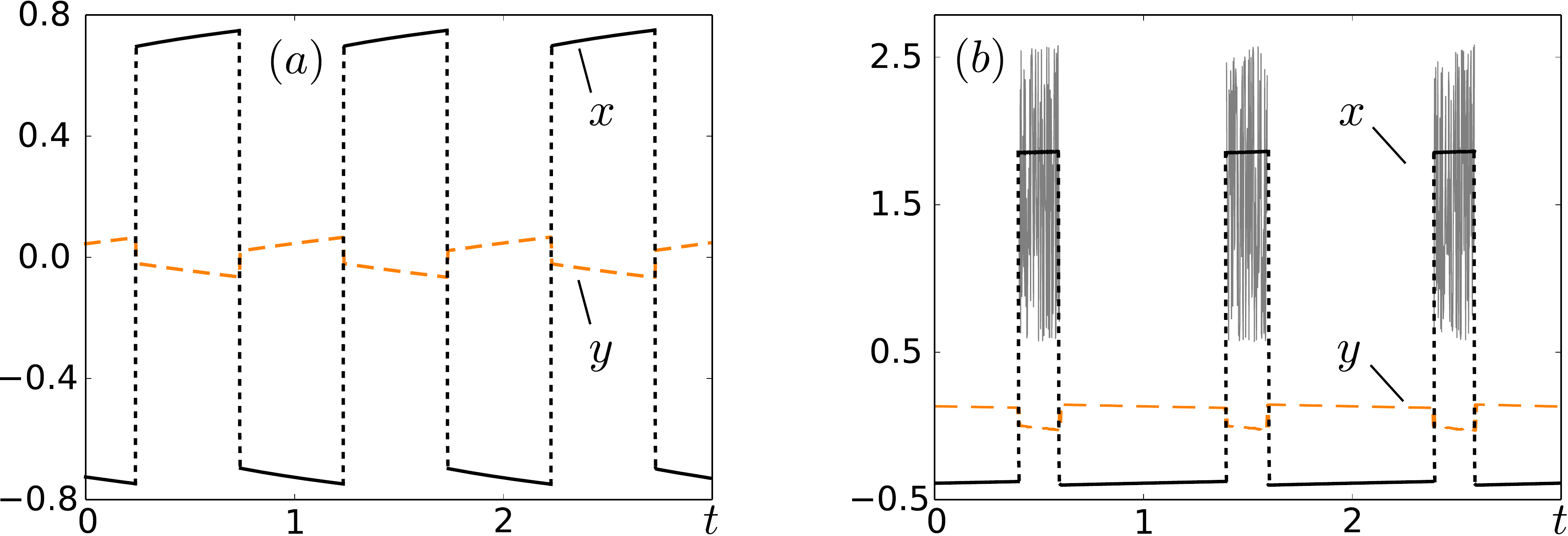}\caption{\label{fig:per-1}(color online) Numerical solution $x$ (black/gray
		solid, discontinuities dotted) and $y$ (orange dashed) of (\ref{eq:def-x-sing})\textendash (\ref{eq:def-y-sing})
		for $\eta=0.3$ and suitable initial condition. Panel (a): period-$1$
		solution for $f(x)=1.2x/(1+x^{4})$. Panel (b) shows irregular behavior
		of the $x$-component (gray) for $f(x)=2x/(1+(0.8-x)^{4})$ whenever
		$|f^{\prime}(x(t))|>1$. The $x$-component of the corresponding unstable
		period-$1$ solution is shown in black.}
\end{figure}


\section{Coexistence with fold-induced slow-fast oscillations\label{sec:coex}}

Section~\ref{sec:balance} revealed that for a given function $f$
and carefully chosen initial conditions close to the balance point,
(\ref{eq:def-x})\textendash (\ref{eq:def-y}) displays delay-induced
switched states. Here, we show that this class of solutions coexists
with another type of solutions oscillating on a much longer timescale
of order $1/\eta$, which can be selected via the choice of initial
conditions.

To begin with, let us choose $f(x)=1.2x/(1+x^{4})$. The balance point
is given by $y(0)=0$ as we have argued in Sec.~\ref{sec:prelim}.
Fig.~\ref{fig:per}(a) shows the $\omega$-limit set of initial
conditions with $y(0)=0$ and two sign changes for various $\eta$.
For $\eta=0.001$, we observe a delay-induced switch state with almost
constant plateaus. The orange dashed curve in Fig.~\ref{fig:per}(a)
shows the fixed points of Eq.~(\ref{eq:def-x-sing}) parameterized
by constant $y=v$, i.e. the fixed points of the iteration $(\chi,v)\mapsto(f(\chi)-v,v)\in\mathbb{R}^{2}$.
We refer to this curve as the critical manifold. We observe that each
sign change in the past, causes the solution to switch to the opposite
branch of the critical manifold (with opposite sign) after some time close to one. Along the critical manifold, $x$ satisfies Eq.~(\ref{eq:x-manifold})
up to an error of order $\varepsilon$, and as a result, $x$ and
$y$ change with rate proportional to $\eta$ on timescale $1$. Therefore,
as we increase $\text{\ensuremath{\eta}},$ the variations along the
plateaus become more pronounced, and the range of $y$ approximately
scales proportional to $\eta;$ compare Fig.~\ref{fig:per}(a). 

\begin{figure}[!]
\noindent \centering{}\includegraphics[width=0.9\linewidth]{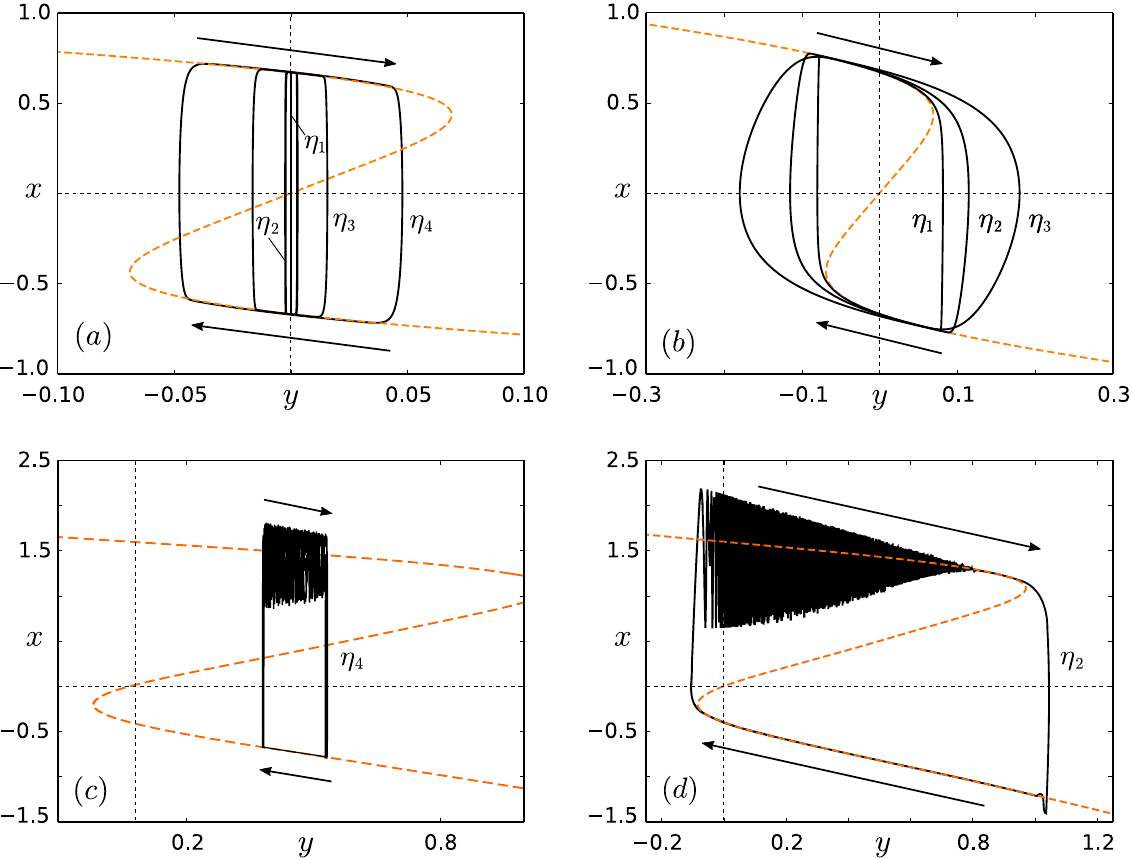}\caption{\label{fig:per}(color online) Delay-induced switched states and ``fold-induced''
slow-fast oscillations coexist in (\ref{eq:def-x})\textendash (\ref{eq:def-y}).
Panels (a\textendash d) show the projection to the $(x,y)$-plane
of the $\omega$-limit set for different initial conditions, choices
of the function $f$ and parameters $\eta_{1}=0.001$, $\eta_{2}=0.01$,
$\eta_{3}=0.1$ and $\eta_{4}=0.3$. The dashed (orange) curve corresponds
to fixed points of the map $(x,y)\protect\mapsto(f(x)-y,y)$, dotted
(black) lines correspond to $x=0$ and $y=0$. Panel (a-b): $\varepsilon=0.01$,
$f(x)=1.2x/(1+x^{4})$, $y(0)=0$; (a): $x_{0}(t)=-1$ for $-1\leq t\leq-1/2,$
$x_{0}(t)=1$ for $-1/2<t\leq0$ (two sign changes), and (b): $x_{0}(t)=-1$
for $-1<t\leq0$ (no sign changes). Panel (c-d): $\varepsilon=0.003,$
$f(x)=2x/(1+(0.8-x)^{4})$ and $y(0)=0.514$; (c): $x_{0}$ as (a),
and (d): $x_{0}$ as (b).}
\end{figure}

When instead we choose $x_{0}$ without sign changes, the picture
is different. Fig.~\ref{fig:per}(b) shows the corresponding $\omega$-limit
set and reveals a different class of solutions to (\ref{eq:def-x})\textendash (\ref{eq:def-y}).
We remark these solutions are reminiscent of slow-fast relaxation
oscillations in systems of ordinary differential equations \cite{Kuehn2015},
the techniques of which can be used to rigorously study those solutions
using geometric theory for semi-flows \cite{Diekmann1995,Bates1998}.
Interestingly here, the switching is not induced by a corresponding
sign change in the past, but rather by a transition through a fold.
Hence, we refer to such solutions as fold-induced slow-fast oscillation
to distinguish them from our solutions of interest. It is out of scope
for this article to give a detailed analysis, but we offer the following
intuition: 
In Sec.~\ref{sec:prelim}, we defined the balance point $y(0)=v$
such that the solution to
\begin{align}
\varepsilon x^{\prime}(t) & =-x(t)-v+f(x(t-1)),\label{eq:def-x-1}
\end{align}
displays exponentially long transients for initial conditions that
have sign changes. Let us put $v=0$ for now. Then, starting with
an initial condition without sign changes, the solutions preserves
its sign \cite{Rost2007}. The $\omega$-limit set in this case may
consist of the nontrivial equilibrium point, a periodic orbit or a
chaotic attractor depending on $f$ \cite{Rost2007}. If the equilibrium
is linearly stable, then the solution converges to it exponentially
fast. The value of the equilibrium is given by the intersection of
the critical manifold with the line $\left\{ y=0,x\lessgtr0\right\} $
in Fig.~\ref{fig:per}(b). For sufficiently small $\eta$, we can
then think of $v=y(t)$ as changing adiabatically, that is the convergence
to the equilibrium for $v\neq0$ is fast, compared to the flow along
the critical manifold. This mechanism gives rise to slow-fast relaxation
oscillations with period of order $1/\eta$, see Fig.~\ref{fig:per}(b).

Let now $f(x)=2x/(1+(x-0.8)^{4})$ and $y(0)\approx0.514$ as numerically
determined in Sec.~\ref{sec:balance}. The stability of an equilibrium
$x(t)=\xi=\text{const}$ of (\ref{eq:def-x-1}) is easily determined.
If $|f^{\prime}(\xi)|>1$, there exists $\bar{\varepsilon}>0$, such
that it is unstable for all $\varepsilon<\bar{\varepsilon}$; and
it is stable for all $\varepsilon>0$, otherwise. At $\varepsilon=\bar{\varepsilon}$
the equilibrium undergoes a Hopf bifurcation with respect to $\varepsilon;$
the proof is analogous to the zero equilibrium in Sec.~\ref{sec:set-up}.
This condition is very much related to Prop.~\ref{prop:per-1}(iv),
for which we obtain delay-induced switched states with incoherent
parts. The corresponding solution to initial conditions without sign
changes is shown in Fig.~\ref{fig:per}(d). This type of solution
has been coined chaotic breathers \cite{Kouomou2005}; we will not
discuss this type of solution in detail. However, we want to emphasize
that they are likely candidates for the omega-limit set of initial
condition with sign changes in (\ref{eq:def-x})\textendash (\ref{eq:def-y})
when is $y(0)$ chosen too far from the balance point.


\section{Outlook and Discussion}

We have investigated delay-induced switched states in systems of
the form (\ref{eq:def-x})-(\ref{eq:def-y}) and how it arises from
the formal limits $\varepsilon\to0$ and/or $\eta\to0$. The phenomenon
itself is more general and has been observed in various slow-fast
systems including time delay, such as laser systems \cite{Weicker2012,Weicker2013}
and in a model of neuronal activity \cite{Erneux2016}.  
In our examples, we restricted ourselves to the case where switching
occurred once per delayed interval, yet multiple switched solution
can be observed as well \cite{Larger2013,Larger2015}. Likewise,
we did not investigate whether delayed-induced switched states can
be reached from non-switched initial conditions. Such a transition
could for example be induced by interaction of the solutions with
a fold, similar to canard solutions \cite{Campbell2009,Krupa2014},
that may lead to a fast increase of the number of sign changes along
a delay interval.
We have introduced the concept of a balance point of (\ref{eq:def-x})-(\ref{eq:def-y}),
and alongside our rigorous results on the ``stabilization'' mechanism
close to the balance point, we included a detailed numerical exploration.
A rigorous proof of the existence of delay-induced states, is beyond
the expository character of this article, and we will address this
question in future works. A promising direction poses the customization
of the asymptotic methods used to study the drift in scalar delayed
feedback equations \cite{Lin1986,Chow1989,Nizette2003,Nizette2004,Grotta-Ragazzo2010,Grotta-Ragazzo1999,Wattis2017}.

\section*{Acknowledgments}

This research was conducted within the framework of CRC 910 founded
by the German Research Foundation (DFG). The authors would like to
thank Laurant Larger, Yuri Maistrenko and Matthias Wolfrum for valuable discussion and suggestions.

\end{document}